\documentclass[captions=tableheading,a4paper, 11pt]{scrartcl}
\usepackage[utf8]{inputenc}
\usepackage[T1]{fontenc}
\usepackage{breqn}
\usepackage{hyperref}
\usepackage{amsthm}
\usepackage{amssymb}
\theoremstyle{plain}
\newtheorem{theorem}{Theorem}
\newtheorem{lemma}[theorem]{Lemma}

\newtheorem{corollary}[theorem]{Corollary}
\newtheorem{proposition}[theorem]{Proposition}
\newtheorem{question}{Question}
\newcommand{\vek}[1]{{\mathbf{#1}}}
\newcommand{\Q}{\mathbf{Q}}
\newcommand{\ZZ}{\mathbb{Z}}
\newcommand{\NN}{\mathbb{N}}
\newcommand{\Reals}{\mathbb{R}}
\newcommand{\TQ}{\tilde{Q}}
\newcommand{\ones}{\mathbf{1}}
\newcommand{\zeroes}{\mathbf{0}}
\author{Jan Snellman\thanks{Department of Mathematics, Linköping University, Sweden; jan.snellman@liu.se}}
\date{}
\title{Some comments on a paper by Azam and Richmond}
\hypersetup{
 pdfauthor={Jan Snellman\thanks{Department of Mathematics, Linköping University, Sweden; jan.snellman@liu.se}},
 pdftitle={Some comments on a paper by Azam and Richmond},
 pdfkeywords={},
 pdfsubject={},
 pdfcreator={Emacs 30.0.93 (Org mode 9.7.19)}, 
 pdflang={English}}

\usepackage{fvextra}

\fvset{%
  commandchars=\\\{\},
  highlightcolor=white!95!black!80!blue,
  breaklines=true,
  breaksymbol=\color{white!60!black}\tiny\ensuremath{\hookrightarrow}}

\usepackage{xcolor}

\providecolor{EfD}{HTML}{f7f7f7}
\providecolor{EFD}{HTML}{28292e}

\usepackage[breakable,xparse]{tcolorbox}
\DeclareTColorBox[]{Code}{o}%
{colback=EfD!98!EFD, colframe=EfD!95!EFD,
  fontupper=\footnotesize\setlength{\fboxsep}{0pt},
  colupper=EFD,
  IfNoValueTF={#1}%
  {boxsep=2pt, arc=2.5pt, outer arc=2.5pt,
    boxrule=0.5pt, left=2pt}%
  {boxsep=2.5pt, arc=0pt, outer arc=0pt,
    boxrule=0pt, leftrule=1.5pt, left=0.5pt},
  right=2pt, top=1pt, bottom=0.5pt,
  breakable}

\usepackage{float}
\floatstyle{plain}
\newfloat{listing}{htbp}{lst}
\newcommand{\listingsname}{Listing}
\floatname{listing}{\listingsname}

\definecolor{EFD}{HTML}{000000}
\definecolor{EfD}{HTML}{ffffff}
\definecolor{EFvp}{HTML}{000000}
\definecolor{EFh}{HTML}{7f7f7f}
\definecolor{EFsc}{HTML}{228b22}
\definecolor{EFw}{HTML}{ff8e00}
\definecolor{EFe}{HTML}{ff0000}
\definecolor{EFl}{HTML}{ff0000}
\definecolor{EFlv}{HTML}{ff0000}
\definecolor{EFhi}{HTML}{ff0000}
\definecolor{EFc}{HTML}{b22222}
\definecolor{EFcd}{HTML}{b22222}
\definecolor{EFs}{HTML}{8b2252}
\definecolor{EFd}{HTML}{8b2252}
\definecolor{EFm}{HTML}{008b8b}
\definecolor{EFk}{HTML}{9370db}
\newcommand{\EFk}[1]{\textcolor{EFk}{#1}} 
\definecolor{EFb}{HTML}{483d8b}
\definecolor{EFf}{HTML}{0000ff}
\definecolor{EFv}{HTML}{a0522d}
\newcommand{\EFv}[1]{\textcolor{EFv}{#1}} 
\definecolor{EFt}{HTML}{228b22}
\definecolor{EFo}{HTML}{008b8b}
\definecolor{EFwr}{HTML}{ff0000}
\definecolor{EFpp}{HTML}{483d8b}
\definecolor{EFOa}{HTML}{0000ff}
\definecolor{EFOb}{HTML}{a0522d}
\definecolor{EFOc}{HTML}{a020f0}
\definecolor{EFOd}{HTML}{b22222}
\definecolor{EFOe}{HTML}{228b22}
\definecolor{EFOf}{HTML}{008b8b}
\definecolor{EFOg}{HTML}{483d8b}
\definecolor{EFOh}{HTML}{8b2252}
\definecolor{EFhn}{HTML}{008b8b}
\newcommand{\EFhn}[1]{\textcolor{EFhn}{#1}} 
\definecolor{EFhq}{HTML}{9370db}
\definecolor{EFhs}{HTML}{008b8b}
\definecolor{EFrda}{HTML}{707183}
\definecolor{EFrdb}{HTML}{7388d6}
\definecolor{EFrdc}{HTML}{909183}
\definecolor{EFrdd}{HTML}{709870}
\definecolor{EFrde}{HTML}{907373}
\definecolor{EFrdf}{HTML}{6276ba}
\definecolor{EFrdg}{HTML}{858580}
\definecolor{EFrdh}{HTML}{80a880}
\definecolor{EFrdi}{HTML}{887070}
\definecolor{EFdiffa}{HTML}{4F894C}
\definecolor{EFdiffc}{HTML}{842879}
\definecolor{EFdiffco}{HTML}{525866}
\definecolor{EFdiffr}{HTML}{99324B}
\definecolor{EFdiffh}{HTML}{398EAC}
\definecolor{EFdifffh}{HTML}{3B6EA8}
\definecolor{EFdiffhh}{HTML}{842879}
\usepackage[backend=bibtex,style=numeric]{biblatex}
\begin{document}

\maketitle

\begin{abstract}
The rational recursion obtained by Azam and Richmond in
\autocite{azam_generating_2023} for generating functions of
\(P_\lambda(y)\), itself a generating function enumerating by length partitions
in the lower ideal \([0,\lambda]\) in the Young lattice, can be easily extended
to a multi-graded version. We demonstrate this and point out the relation to
enumerating plane partitions with two rows. By means of this simple observation,
we can relate Azam and Richmond's result to those obtained by Andrews and Paule
in \autocite{andrews_macmahons_2007} using MacMahons
\(\Omega\)-operator.
\end{abstract}
\section{Introduction}
\label{sec:orgea4483b}

In \autocite{azam_generating_2023} Azam and Richmond studied the
rank-generating function

\[P_\lambda(y) = \sum_{\mu \in [0,\lambda]} y^{|\mu|}\]

of the lower order ideal \([0,\lambda]\) in the Young lattice. They obtained a
rational recursion for

\[Q_k(\vek{x},y) = \sum_{\lambda \in \Lambda(k)} P_\lambda(y) \vek{x}^{\lambda}\]

where \(\Lambda(k)\) denotes the set of partitions with length \(k\). They
concluded that \(Q_k\) is a rational function, with denominator

\[
D_k(x_1,\dots,x_k,y) =
\prod_{m=1}^k \prod_{j=0}^m (1-y^j \prod_{\ell=1}^m x_\ell).
\]

These results were used to establish asymptotics for the average cardinality of
lower order ideals \([0,\lambda]\) of partitions \(\lambda\) of rank \(n\).
\section{Multigradings, pairs of partitions, and plane partitions with two rows}
\label{sec:orgdc34a86}
\subsection{The generating functions \(Q_k\) and \(\TQ_k\)}
\label{sec:org603d98f}
Let us define

\[Q_k(\vek{x},\vek{y}) =
\sum_{\emptyset \le \mu \le \lambda \in
\Lambda(k)} \vek{y}^\mu \vek{x}^{\lambda}.
\]

Then specializing \(y_1=y_2= \cdots= y_k = y\) we get back the previous
\(Q_k(\vek{x},y)\). However, the multigraded version can be interpreted as the
generating function of plane partitions with at most two rows, where the top row,
representing \(\lambda\), has \(\lambda_k >0\).
Introducing

\[\TQ_k(\vek{x},\vek{y}) =
\sum_{\emptyset \le \mu \le \lambda \in \Lambda(\le k)}
\vek{y}^\mu \vek{x}^{\lambda}\]

where \(\Lambda(\le k)\) denotes partitions with length at most \(k\),
we have that

\[\TQ_k =\sum_{j=0}^k Q_k\]

and that

\[Q_k = \TQ_k - \TQ_{k-1}.\]
\subsection{Cones, hyperplanes, and polytopes}
\label{sec:orgef04a87}
The generating function \(\TQ_k\) enumerates plane partitions
contained in a 2-by-\(k\) box. Explicitly, the inequalities that
the integer vectors \((\lambda,\mu) \in \ZZ^k \times \ZZ^k\) has to
satisfy are as follows:

\begin{align}
\lambda_i - \lambda_j & \ge 0 \qquad \forall i < j \\
\mu_i - \mu_j & \ge 0 \qquad \forall i < j \\
\lambda_i - \mu_i & \ge 0 \qquad \forall i  \\
\lambda_i & \ge 0 \qquad \forall i \\
\mu_i & \ge 0 \qquad \forall i
\end{align}

We let \(C=C_k \subset \Reals^k \times \Reals^k\)
denote the rational pointed polyhedral cone cut out in affine space by the
above inequalities, and let \(A_k\)
be its ``integer transform'', that is to say, the affine monoid
\(C_k \cap (\ZZ^k \times \ZZ^k)\).
\subsubsection{The case \(k=2\)}
\label{sec:org2c3cd4b}
For instance, when \(k=2\), the plane partitions in a \(2 \times 2\)-box
are
\begin{displaymath}
  \begin{pmatrix}
    \lambda_1 & \lambda_2 \\
    \mu_1 & \mu_2
  \end{pmatrix}
\qquad
\lambda_1 \ge \lambda_2 \ge \mu_2 \ge 0, \, \lambda_1 \ge \mu_1 \ge \mu_2 \ge 0.
\end{displaymath}

The corresponding integer transform is \(\TQ_2(x_1,x_2,y_1,y_2)\);
to get the plane partitions enumerated by \(Q_2((x_1,x_2,y_1,y_2)\)
we add the extra iequality \(\lambda_2 >0\).
The resulting polyhedron has \(C_2\) as its recession cone.

The cone
\(C_2 \subset \Reals^2 \times \Reals^2\) has 5 extremal rays.
We can calculate these using
\autocites{sagemath}[][]{normaliz}.
\begin{table}[htbp]
\caption{\label{tab:cone2rays}Generating rays of plane partitions inside a 2 by 2 box}
\centering
\begin{tabular}{rl}
0 & (1, 0, 0, 0)\\
1 & (1, 0, 1, 0)\\
2 & (1, 1, 0, 0)\\
3 & (1, 1, 1, 0)\\
4 & (1, 1, 1, 1)\\
\end{tabular}
\end{table}
\subsubsection{General \(k\)}
\label{sec:org7bd7994}
For a general \(k\),
we note that all extremal rays of \(C=C_k\) intersect the affine hyperplane
\[H=\{(\lambda_1,\dots,\lambda_k,\mu_1,\dots,\mu_k)\, : \, \lambda_1=1\}\]
in lattice points. Call the set of these points \(S_k\).
Let \(P=P_k\) be the intersection  \(C \cap H\).
Let \[T_k = P \cap (\ZZ^k \times \ZZ^k).\]
Recall that we introduced the affine monoid
\[A=A_k = C_k \cap (\ZZ^k \times \ZZ^k)\]
whose generating function is \(\TQ_k\).

\begin{lemma}
Let \(C_k\), \(A_k\), \(P_k\), \(S_k\), \(T_k\) be as above. Then
\begin{enumerate}
\item The cone \(C\) is the disjoint union
\[C = \cup_{t \ge 0} tP\]
of dilations of \(P\).
\item \(S_k = T_k\).
\item Denote the vector of length \(r\) consisting
of all ones by \(\ones^r\), and the vector of length \(r\) consisting of all
zeroes by \(\zeroes^r\). Put
\begin{equation}
\label{eqn:Sk}
U_k=\{(\ones^a,\zeroes^b,\ones^c,\zeroes^d) \, : \,
a+b=c+d=k, a \ge c, a \ge 1\}.
\end{equation}
Then \(S_k=T_k = U_k\).
\item The polytope \(P\) is the convex hull of \(S_k\).
\item \(\TQ_k\) is the multigraded Ehrhart series of \(P\).
\item Let \(D_k = \prod_{\vek{r} \in S_k} \left(1-(\vek{x}\vek{y})^{\vek{r}}\right)\).
Then \(\TQ_k \times D_k\) is a polynomial.
\item \(S_k\) form a Hilbert basis for the affine monoid \(A_k\).
\end{enumerate}
\label{hilbert_basis}
\end{lemma}
\begin{proof}
Let \((\lambda,\mu)\) be a plane partition in \(A_k\). If \((\lambda,\mu)\) is
nonzero, then \(\lambda_1 \ge 1\). Let \((s(\lambda),s(\mu))\) be the support of
the pair; here \(s(\lambda)(i) = 1\) if \(\lambda_i > 0\), and zero otherwise.
Then it is easy to see that \((s(\lambda),s(\mu)) \in U_k\).
Furthermore,

\[(\lambda,\mu) -(s(\lambda),s(\mu)) \in A_k.\]

Thus, every element in \(A_k\) is expressible as a sum of elements in \(U_k\).

Elements in \(T_k\) are irreducible; if the partition \((\ones^a,\zeroes^b)\)
is to written as a sum of elements in \(\NN^k \times \NN^k\), one of the
summands would have to start with a zero --- but this is impossible.

By Gordan's lemma (see for instance \autocite{gubeladze_affine_2009}) we have
that the Hilbert basis of \(A_k\) consists of the irreducible elements in the
monoid. Any element in \((\lambda,\mu) \in A_k\) with \(\lambda_1 > 1\) can be
written as

\[
(\lambda,\mu) = (s(\lambda),s(\mu)) +
\left( (\lambda,\mu) - (s(\lambda),s(\mu)) \right)
\]

and is thus reducible. Hence, the Hilbert basis
consists precisely of \(T_k\), and this set is equal to \(U_k\) and \(S_k\).
\end{proof}

For a simplicial rational cone, the generating function has numerator 1,
and denominator given by the extremal rays. Our cone \(C\) is not simplicial, though;
it has more generators than the embedding dimension \(2k\).
Thus the numerator is some mulitvariate polynomial. However, from general theory
\autocites{gubeladze_affine_2009}[][]{schrijver_theory_2011}
it follows that
\begin{corollary}
The denominator of \(\TQ_k\), and hence of \(Q_k\), is precisely
\(D_k\)
\end{corollary}

Specializing \(y_1=\cdots=y_k = y\) we recover Proposition 15 of
\autocite{azam_generating_2023}.
In the multigraded
case there can be no cancellation between the numerator and the denominator of
\(Q_k\), so we can assert that this \(D_k\) is \emph{the} denominator, not just
divisible by the denominator.
\subsection{Calculating \(\TQ_k\) by triangulating \(C_k\)}
\label{sec:org0882764}
\subsubsection{\(k=2\)}
\label{sec:org75f652d}
Let us consider \(C_2\) again. It lives in \(\Reals^2 \times \Reals^2\) but,
as was shown in Table \ref{tab:cone2rays}, it
is spanned by 5 extremal rays, hence it is it not simplicial. We can,
however, triangulate it into a union of simplicial cones.
Sagemath + Normaliz
gives a triangulation, shown in Table \ref{c2-tri} (rows indicate subsets of rays).

\begin{table}[htbp]
\caption{\label{c2-tri}Triangulation of \(C_{2}\)}
\centering
\begin{tabular}{rrrr}
0 & 1 & 2 & 4\\
1 & 2 & 3 & 4\\
\end{tabular}
\end{table}

So \(C=C_2= K_1 \cup K_2\), where \(K_1,K_2\) and \(K_3=K_1 \cap K_2\)
are rational simplicial cones. \(K_3\) is generated by the intersection of
the generating rays of \(K_1\) and of \(K_2\), that is to say, by \(r_1,r_1,r_4\).

A rational polyhedral simplicial cone generated by the rays \(\vek{r}\)
will have generating function
\[\frac{1}{\prod_{\vek{r}} (1-(\vek{x}\vek{y})^{\vek{r}}))}.\]

Hence, by inclusion-exclusion,

\begin{multline*}
\frac{N}{(1-(\vek{x}\vek{y})^{r_0})
(1-(\vek{x}\vek{y})^{r_1})
(1-(\vek{x}\vek{y})^{r_2})
(1-(\vek{x}\vek{y})^{r_3})
(1-(\vek{x}\vek{y})^{r_4})}
\\ =
\frac{1}{(1-(\vek{x}\vek{y})^{r_0})
(1-(\vek{x}\vek{y})^{r_1})
(1-(\vek{x}\vek{y})^{r_2})
(1-(\vek{x}\vek{y})^{r_4})}
\\ +
\frac{1}{(1-(\vek{x}\vek{y})^{r_1})
(1-(\vek{x}\vek{y})^{r_2})
(1-(\vek{x}\vek{y})^{r_3})
(1-(\vek{x}\vek{y})^{r_4})}
\\ -
\frac{1}{(1-(\vek{x}\vek{y})^{r_1})
(1-(\vek{x}\vek{y})^{r_2})
(1-(\vek{x}\vek{y})^{r_4})}
\end{multline*}

hence

\[
N = (\vek{x}\vek{y})^{r_3} + (\vek{x}\vek{y})^{r_0} - (\vek{x}\vek{y})^{r_0}(\vek{x}\vek{y})^{r_3},
\]

which evaluates to

\begin{multline*}
( -x\textsubscript{0} y\textsubscript{0} x\textsubscript{1} + 1 ) + ( -x\textsubscript{0} + 1 ) - ( x\textsubscript{0}\textsuperscript{2} y\textsubscript{0} x\textsubscript{1} - x\textsubscript{0} y\textsubscript{0} x\textsubscript{1} - x\textsubscript{0} + 1 ) = -x\textsubscript{0}\textsuperscript{2} y\textsubscript{0} x\textsubscript{1} + 1
\label{inclusion-exclusion}
\end{multline*}
\subsubsection{\(k=3\)}
\label{sec:orgf22186a}
\paragraph{Plane partitions}
\label{sec:org05c217f}
For \(k=3\) the
plane partitions  are
\begin{displaymath}
\begin{pmatrix}
\lambda_1 & \lambda_2 & \lambda_3\\
\mu_1 & \mu_2 & \mu_3
\end{pmatrix}
\end{displaymath}
with inequalites ensuring that the entries are non-negative
and non-increasing in rows and columns.
\paragraph{Extremal rays}
\label{sec:org0232ce9}
There are now 9 extremal rays, generating the cone \(C=C_3 \subset \Reals^3 \times \Reals^3\).
\begin{table}[htbp]
\caption{\label{cone-2-3-rays}Extremal rays of plane partitions with 2 rows and 3 columns}
\centering
\begin{tabular}{rl}
0 & (1, 0, 0, 0, 0, 0)\\
1 & (1, 0, 0, 1, 0, 0)\\
2 & (1, 1, 0, 0, 0, 0)\\
3 & (1, 1, 0, 1, 0, 0)\\
4 & (1, 1, 0, 1, 1, 0)\\
5 & (1, 1, 1, 0, 0, 0)\\
6 & (1, 1, 1, 1, 0, 0)\\
7 & (1, 1, 1, 1, 1, 0)\\
8 & (1, 1, 1, 1, 1, 1)\\
\end{tabular}
\end{table}
\paragraph{Triangulation}
\label{sec:orga2fbb6f}
A (regular) triangulation of the cone, with rays numbered as in Table \ref{cone-2-3-rays}, is shown in
Table \ref{cone-3-2-tri}.
\begin{table}[htbp]
\caption{\label{cone-3-2-tri}Triangulation of cone of plane partitions with 2 rows and 3 columns, rows are subcones}
\centering
\begin{tabular}{rrrrrr}
0 & 1 & 2 & 4 & 5 & 8\\
0 & 1 & 4 & 5 & 7 & 8\\
1 & 2 & 3 & 4 & 5 & 8\\
1 & 3 & 4 & 5 & 6 & 8\\
1 & 4 & 5 & 6 & 7 & 8\\
\end{tabular}
\end{table}
\subsubsection{General \(k\)}
\label{sec:org8728ec2}
It is feasible to use inclusion-exclusion to find \(\TQ_3\),
the generating function of the cone \(C_3\).
However, this is not an efficient way of calculating \(\TQ_k\) for general \(k\).
The number of extremal rays of \(C_k\) is, as we shown, equal to one less
the number of plane partitions
inside a \(2 \times k \times 1\) box. From \autocite{macmahon_combinatory_1915}, this number
is
\begin{math}
\binom{2+k}{2} -1.
\end{math}
The number of simplicial subcones in the triangulation grows swiftly; it is equal to the
Catalan number:
\begin{table}[htbp]
\caption{\label{cones-in-tri}nr  of cones in triangulation of C}
\centering
\begin{tabular}{rrrr}
k & dim(C) & nr rays & nr cones in tri\\
2 & 4 & 5 & 2\\
3 & 6 & 9 & 5\\
4 & 8 & 14 & 14\\
5 & 10 & 20 & 42\\
6 & 12 & 27 & 132\\
7 & 14 & 35 & 429\\
8 & 16 & 44 & 1430\\
9 & 18 & 54 & 4862\\
\end{tabular}
\end{table}
\section{The rational recursion of Azam and Richmond}
\label{sec:org272f879}
\subsection{Original version}
\label{sec:org78223be}
We state the main result of \autocite{azam_generating_2023}. Recall
that their \(Q_k\) is multi-graded in \(\vek{x}\) but simply-graded in \(y\), so
depends on \(k+1\) variables.
\begin{theorem}[Azam and Richmond Thm 1]
Let \(p_k = x_1 \cdots x_k\), and for a sequence of parameters \(Z=(z_1,\dots,z_{k+1})\), let
\[
Q_k(Z) = Q_k(z_1,\dots,z_k).
\]

\begin{itemize}
\item If \(Z=(x_1,\dots,x_k,y)\), then denote \(Q_k=Q_k(Z)\).
\item For \(0 < r \le k\), we put \(Z_r = (y^rp_{r+1},x_{r+2},x_{r+3},\dots,x_k,y)\).
\end{itemize}

Then \(Q_0=1\) and for \(k \ge 1\) we
have
\begin{equation}\label{eq:main-recursion}
(1-p_k)Q_k =
x_kQ_{k-1} +
\sum_{0 \le i < r \le k} \left( \frac{y^{r}p_{k}}{1-y^{r}p_{r}} \right) Q_{{k-r}}(Z_{r}) \cdot Q_{i}
\end{equation}
In particular, \(Q_k\) is a rational function in the variables
\(x_1,\dots,x_k,y\).
\label{azam-richmond-mainthm}
\end{theorem}

They go on the prove
\begin{proposition}[Azam and Richmond Proposition 15]
Let \(p_k = x_1 \cdots x_k\), and
\(D_k = D_k(x_1,\dots,x_k,y) = \prod_{m=1}^k \prod_{j=0}^m (1-y^jp_m)\).
Then \(Q_k \cdot D_k\) is a polynomial.
\label{a-r-d}
\end{proposition}

As we have seen, this latter results is a straight-forward consequence
of classification of the generating rays of \(C_k\).

The numerators \(N_k\) of \(Q_k = N_k/D_k\) are given below, for \(k=1,2\).
We show the multigraded case; the \(y\)-simplygraded case, (as studied by Azam and Richmond)
can be recovered by setting the different \(y_i\)'s to \(y\).

\(k=1\):
\begin{math}
-x\textsubscript{1}\textsuperscript{2} y\textsubscript{1} + x\textsubscript{1} y\textsubscript{1} + x\textsubscript{1}
\label{}
\end{math}
.

\(k=2\):
\begin{dmath*}
x\textsubscript{1}\textsuperscript{3} y\textsubscript{1}\textsuperscript{2} x\textsubscript{2}\textsuperscript{3} y\textsubscript{2} - x\textsubscript{1}\textsuperscript{2} y\textsubscript{1}\textsuperscript{2} x\textsubscript{2}\textsuperscript{2} y\textsubscript{2} - x\textsubscript{1}\textsuperscript{2} y\textsubscript{1} x\textsubscript{2}\textsuperscript{2} y\textsubscript{2} - x\textsubscript{1}\textsuperscript{2} y\textsubscript{1} x\textsubscript{2}\textsuperscript{2} - x\textsubscript{1}\textsuperscript{2} y\textsubscript{1} x\textsubscript{2} + x\textsubscript{1} y\textsubscript{1} x\textsubscript{2} y\textsubscript{2} + x\textsubscript{1} y\textsubscript{1} x\textsubscript{2} + x\textsubscript{1} x\textsubscript{2}
\label{}
\end{dmath*}

\(k=3\):
\begin{dmath*}
x\textsubscript{1}\textsuperscript{6} y\textsubscript{1}\textsuperscript{4} x\textsubscript{2}\textsuperscript{5} y\textsubscript{2}\textsuperscript{2} x\textsubscript{3}\textsuperscript{4} y\textsubscript{3} - x\textsubscript{1}\textsuperscript{5} y\textsubscript{1}\textsuperscript{4} x\textsubscript{2}\textsuperscript{4} y\textsubscript{2}\textsuperscript{2} x\textsubscript{3}\textsuperscript{3} y\textsubscript{3} - x\textsubscript{1}\textsuperscript{5} y\textsubscript{1}\textsuperscript{3} x\textsubscript{2}\textsuperscript{4} y\textsubscript{2}\textsuperscript{2} x\textsubscript{3}\textsuperscript{3} y\textsubscript{3} - x\textsubscript{1}\textsuperscript{4} y\textsubscript{1}\textsuperscript{3} x\textsubscript{2}\textsuperscript{4} y\textsubscript{2}\textsuperscript{2} x\textsubscript{3}\textsuperscript{4} y\textsubscript{3} - x\textsubscript{1}\textsuperscript{5} y\textsubscript{1}\textsuperscript{3} x\textsubscript{2}\textsuperscript{4} y\textsubscript{2} x\textsubscript{3}\textsuperscript{3} y\textsubscript{3} - x\textsubscript{1}\textsuperscript{5} y\textsubscript{1}\textsuperscript{3} x\textsubscript{2}\textsuperscript{4} y\textsubscript{2} x\textsubscript{3}\textsuperscript{3} - x\textsubscript{1}\textsuperscript{5} y\textsubscript{1}\textsuperscript{3} x\textsubscript{2}\textsuperscript{4} y\textsubscript{2} x\textsubscript{3}\textsuperscript{2} + x\textsubscript{1}\textsuperscript{4} y\textsubscript{1}\textsuperscript{3} x\textsubscript{2}\textsuperscript{3} y\textsubscript{2}\textsuperscript{2} x\textsubscript{3}\textsuperscript{2} y\textsubscript{3} + x\textsubscript{1}\textsuperscript{3} y\textsubscript{1}\textsuperscript{3} x\textsubscript{2}\textsuperscript{3} y\textsubscript{2}\textsuperscript{2} x\textsubscript{3}\textsuperscript{3} y\textsubscript{3} + x\textsubscript{1}\textsuperscript{4} y\textsubscript{1}\textsuperscript{3} x\textsubscript{2}\textsuperscript{3} y\textsubscript{2} x\textsubscript{3}\textsuperscript{2} y\textsubscript{3} + x\textsubscript{1}\textsuperscript{3} y\textsubscript{1}\textsuperscript{2} x\textsubscript{2}\textsuperscript{3} y\textsubscript{2}\textsuperscript{2} x\textsubscript{3}\textsuperscript{3} y\textsubscript{3} + x\textsubscript{1}\textsuperscript{4} y\textsubscript{1}\textsuperscript{3} x\textsubscript{2}\textsuperscript{3} y\textsubscript{2} x\textsubscript{3}\textsuperscript{2} + x\textsubscript{1}\textsuperscript{4} y\textsubscript{1}\textsuperscript{2} x\textsubscript{2}\textsuperscript{3} y\textsubscript{2} x\textsubscript{3}\textsuperscript{2} y\textsubscript{3} + x\textsubscript{1}\textsuperscript{3} y\textsubscript{1}\textsuperscript{2} x\textsubscript{2}\textsuperscript{3} y\textsubscript{2} x\textsubscript{3}\textsuperscript{3} y\textsubscript{3} + x\textsubscript{1}\textsuperscript{4} y\textsubscript{1}\textsuperscript{2} x\textsubscript{2}\textsuperscript{3} y\textsubscript{2} x\textsubscript{3}\textsuperscript{2} + x\textsubscript{1}\textsuperscript{3} y\textsubscript{1}\textsuperscript{2} x\textsubscript{2}\textsuperscript{3} y\textsubscript{2} x\textsubscript{3}\textsuperscript{3} + x\textsubscript{1}\textsuperscript{4} y\textsubscript{1}\textsuperscript{2} x\textsubscript{2}\textsuperscript{3} x\textsubscript{3}\textsuperscript{2} + x\textsubscript{1}\textsuperscript{3} y\textsubscript{1}\textsuperscript{2} x\textsubscript{2}\textsuperscript{3} y\textsubscript{2} x\textsubscript{3}\textsuperscript{2} - x\textsubscript{1}\textsuperscript{2} y\textsubscript{1}\textsuperscript{2} x\textsubscript{2}\textsuperscript{2} y\textsubscript{2}\textsuperscript{2} x\textsubscript{3}\textsuperscript{2} y\textsubscript{3} + x\textsubscript{1}\textsuperscript{3} y\textsubscript{1}\textsuperscript{2} x\textsubscript{2}\textsuperscript{3} y\textsubscript{2} x\textsubscript{3} - x\textsubscript{1}\textsuperscript{3} y\textsubscript{1}\textsuperscript{2} x\textsubscript{2}\textsuperscript{2} y\textsubscript{2} x\textsubscript{3} y\textsubscript{3} - x\textsubscript{1}\textsuperscript{2} y\textsubscript{1}\textsuperscript{2} x\textsubscript{2}\textsuperscript{2} y\textsubscript{2} x\textsubscript{3}\textsuperscript{2} y\textsubscript{3} - x\textsubscript{1}\textsuperscript{2} y\textsubscript{1}\textsuperscript{2} x\textsubscript{2}\textsuperscript{2} y\textsubscript{2} x\textsubscript{3}\textsuperscript{2} - x\textsubscript{1}\textsuperscript{2} y\textsubscript{1} x\textsubscript{2}\textsuperscript{2} y\textsubscript{2} x\textsubscript{3}\textsuperscript{2} y\textsubscript{3} - x\textsubscript{1}\textsuperscript{2} y\textsubscript{1}\textsuperscript{2} x\textsubscript{2}\textsuperscript{2} y\textsubscript{2} x\textsubscript{3} - x\textsubscript{1}\textsuperscript{2} y\textsubscript{1} x\textsubscript{2}\textsuperscript{2} y\textsubscript{2} x\textsubscript{3}\textsuperscript{2} - x\textsubscript{1}\textsuperscript{2} y\textsubscript{1} x\textsubscript{2}\textsuperscript{2} y\textsubscript{2} x\textsubscript{3} - x\textsubscript{1}\textsuperscript{2} y\textsubscript{1} x\textsubscript{2}\textsuperscript{2} x\textsubscript{3}\textsuperscript{2} - x\textsubscript{1}\textsuperscript{2} y\textsubscript{1} x\textsubscript{2}\textsuperscript{2} x\textsubscript{3} + x\textsubscript{1} y\textsubscript{1} x\textsubscript{2} y\textsubscript{2} x\textsubscript{3} y\textsubscript{3} - x\textsubscript{1}\textsuperscript{2} y\textsubscript{1} x\textsubscript{2} x\textsubscript{3} + x\textsubscript{1} y\textsubscript{1} x\textsubscript{2} y\textsubscript{2} x\textsubscript{3} + x\textsubscript{1} y\textsubscript{1} x\textsubscript{2} x\textsubscript{3} + x\textsubscript{1} x\textsubscript{2} x\textsubscript{3}
\label{}
\end{dmath*}
\subsection{Multigraded version}
\label{sec:org89eb75f}
We come to the main purpose of this note: the rational recursion of \(Q_k\) works
multigradedly!

\begin{corollary}
For \(i,r,k \ge 0\), define
\begin{itemize}
\item \(Q_k=Q_k(\vek{x},\vek{y})\)
\item \(p_r = x_1 \cdots x_r\)
\item \(q_r = y_1 \cdots y_r\)
\item \(\hat{Z}_{r,k} = (p_{r+1}q_{r},x_{r+2},\dots,x_k,y_{r+1},\dots,y_k)\)
\item \(R_{i,r,k} = \frac{p_kq_r}{(1-p_k)(1-p_rq_r)}Q_i Q_{k-r}(\hat{Z}_{r,k})\)
\end{itemize}

Then \(Q_0=1\) and for \(k>0\)
\begin{equation}
\label{eqn:mg}
Q_k = \frac{x_k Q_{k-1}}{1-p_k} +
\sum_{0 \le i < r \le k} R_{i,r,k}
\end{equation}
\label{cor-main}
\end{corollary}

\begin{proof}[Proof (sketch)]
The difference to the original theorem is that \(Q_k=Q_k(\vek{x},\vek{y})\)
is a function of \(2k\) variables whereas \(\Q_k(\vek{x},y)\) is a function of
\(k+1\). Furthermore, the substitution in \(Q_{k-r}\) is
refined to
\[
Q_{k-r}(x_1\cdots x_{r+1} \cdot y_1 \cdots y_r,
x_{r+2},\dots,x_k,y_{r+1},\dots,y_k)
\]
rather than
\[
Q_{k-r}(x_1 \cdots x_{r+1} \cdot y^{r},x_{r+2},\dots,x_k,y).
\]
The variuos lemmas and propositions in Section 2 of
\autocite{azam_generating_2023}
that prove the recursion are based on bijections, and can be modified so to
work multigradedly. Specifically:

\begin{itemize}
\item Replace
\(P_\lambda(y) = \sum_{\mu \in [\emptyset, \lambda]}y^{|\mu|}\)
with
\(P_{\lambda}(\vek{y}) = \sum_{\mu \in [\emptyset,\lambda]} \vek{y}^{\mu}\)
\item Replace
\(P_{\mu,\lambda}(y) = \sum_{\nu \in [\mu,\lambda]} y^{|\nu|}\)
with
\(P_{\mu,\lambda}(\vek{y}) = \sum_{\nu \in [\mu,\lambda]} \vek{y}^{\nu}\)
\item Replace
\[Q_{k,m}(\vek{x},y) =
  \sum_{\lambda \in \Lambda(k,m)} P_{\lambda}(y) \vek{x}^\lambda
  \]
with
\[
  Q_{k,m}(\vek{x},\vek{y}) =
  \sum_{\lambda \in \Lambda(k,m)} P_{\lambda}(\vek{y}) \vek{x}^\lambda
  \]
\item Proposition 12: Replace \(y^k\) with \(y_1 \cdots y_k\).
\item Proposition 14: Also replace
\[
   Q_{k-r,m-1}(y^rp_{r+1},x_{r+2},\dots,x_k,y)
   \text{ with }
   Q_{k-r,m-1}(q_rp_{r+1},x_{r+2},\dots,x_k,y_{r+1},\dots,y_k)
  \]
\item Lemma 13, Theorem 1: Do the above replacements.
\end{itemize}
\end{proof}
\section{Relation to prior work by Andrews and Paule and MacMahon}
\label{sec:orgeaeac29}
\subsection{Geometric interpretation of the rational recursion}
\label{sec:orga99ac89}
The rational recursion above yields an efficient way of calculating
\(Q_{\ell}\), and hence \(\TQ_{\ell}\). Explicitly,
\begin{align*}\label{eq:TQrec}
  \TQ_{\ell}
  & = \sum_{k=0}^{\ell} Q_{k} \\
  & = \sum_{k=0}^{\ell} \left(
    \frac{x_{k} Q_{k-1}}{1-p_{k}} +
    \sum_{0 \le i < r \le k} R_{i,r,k} \right) \\
  & = \sum_{k=0}^{\ell} \left(
    \frac{x_{k} Q_{k-1}}{1-p_{k}} +
    \sum_{0 \le i < r \le k}
    \frac{p_kq_r}{(1-p_k)(1-p_rq_r)}Q_i Q_{k-r}(\hat{Z}_{r,k})
    \right) 
\end{align*}
This is a description how to slice up the affine monoid \(A_{\ell}\)
into disjoint pieces; \(Q_k\) enumerates lattice points in
\(C \cap H^+\), \(C\) being the polyhedral cone, and \(H+\) the open
halfspace \(\lambda_k > 0\).
The term \(Q_{k-1} \frac{x_k}{1-p_k}\) enumerates lattice points
in the translation of the projection of \(C\) in a certain direction,
et cetera.
It is not a triangulation of
\(C\) into subcones, nor is it a ``disjoint decomposition'' as
is computed by Normaliz; it is much more complicated.
\subsection{Generating functions for plane partitions in a box using the Omega operator}
\label{sec:org1f27b6b}
In a series of papers, out of all which we will refer to
\autocite{andrews_macmahons_2007}, Andrews and Paule
revisits MacMahon's method of partition analysis. They define
\[
p_{m,n}(X) = \sum_{a_{i,j} \in P_{m,n}}
x_{1,1}^{a_{1,1}} \cdots x_{m,n}^{a_{m,n}}
\]
where \(P_{m,n}\) consists of all \(m \times n\) matrices \((a_{i,j})\)
over non-negative integers \(a_{i,j}\) such that
\(a_{i,j} \ge a_{i,j+1}\) and \(a_{i,j} \ge a_{i+1,j}\). Putting \(m=2\), we
get our objects of interest.

They then (pages 650-651) illustrate MacMahon's method using his \(\Omega\)
operator by calculating \(p_{2,2}(X)\). This is of course the same as
\(\TQ_2\).

We replicate their calculations using the Omega package (written by Daniel Krenn) in Sagemath.
We could also have used the mathematica package \autocite{andrews_macmahons_2001}
by Andrews et al,  or the Maple package \autocite{zeilberger_lindiophantustxt_2001}.
by Doron Zeilberger.
\begin{Code}
\begin{Verbatim}
\color{EFD}L.<mu11,mu12,l11,l21,x11,x12,x21,x22> = LaurentPolynomialRing(ZZ)
\EFv{p22setup} = [\EFhn{1}-x11*l11*mu11, \EFhn{1}-x21*l21/mu11, \EFhn{1}-x12*mu12/l11, \EFhn{1}-x22/(l21*mu12)]
\EFv{p22} = MacMahonOmega(l21,
                    MacMahonOmega(l11,
                                  MacMahonOmega(mu12,
                                                MacMahonOmega(mu11, \EFhn{1}, p22setup))))
[(t[\EFhn{0}],t[\EFhn{1}]) \EFk{for} t \EFk{in} p22]
\end{Verbatim}
\end{Code}

\phantomsection
\label{}
\begin{verbatim}
[(-x11^2*x12*x21 + 1, 1),
 (-x11 + 1, -1),
 (-x11*x12 + 1, -1),
 (-x11*x12*x21*x22 + 1, -1),
 (-x11*x21 + 1, -1),
 (-x11*x12*x21 + 1, -1)]
\end{verbatim}

We recognize the numerator and denominator of \(\Q_2\), with renamed variables.

The most interesting part, for us,
in \autocite{andrews_macmahons_2007},
is their Lemma 2.3, which provides a recursion for plane partitions in an
\(m \times n\) box. Specialising to \(m=2\) we get

\begin{corollary}[Andrews and Paule Lemma 2.3]
\begin{multline}\label{eq:APrec}
p_{2,n+1}
\begin{pmatrix}
  x_{{1,1}} & \cdots & x_{1,n} & x_{1,n+1} \\
  x_{{2,1}} & \cdots & x_{2,n} & x_{2,n+1}
\end{pmatrix}
=
\left(
1 - x_{{1,n+1}}x_{{2,n+1}} \prod_{{1 \le i \le 2, 1 \le j \le n}} x_{i,j}
\right)^{{-1}}
\\
\times
\Omega_{\ge} \quad p_{{2,n}}
\begin{pmatrix}
  x_{{1,1}} & \cdots & x_{1,n-1} & \lambda_{0}x_{1,n} \\
  x_{{2,1}} & \cdots & x_{2,n-1} & \lambda_{1}x_{2,n}
\end{pmatrix}
\\
\times
\frac{1}{(1-\frac{x_{1,n+1}}{\lambda_{0}}) (1-\frac{x_{1,n+1}x_{{2,n+1}}}{\lambda_{0}\lambda_{1}}) }
\end{multline}
\end{corollary}

Without going into details regarding the \(\Omega\) operator, we will mention that
it operates of formal Laurent polynomials and transforms the expression under
its purvey so that the ``spurious'' \(\lambda\) variables
(not related to partitions, we are using Andrews' and Paule's notations here)
gets eliminated, and what is left is the desired generating function.

\begin{question}
Is there a relation between the ``rational recursion'' (\ref{eqn:mg})
and Andrews' and Paule's Lemma 2.3?
\end{question}
\section{References}
\label{sec:org51b95ef}
\printbibliography
\end{document}